\documentclass[12pt,twoside]{amsart}

\usepackage{hyperref}
\usepackage{amsthm, amsmath, amscd, amssymb,centernot}
\usepackage[all]{xy}
\usepackage[T1]{fontenc}
\usepackage[left=2.5cm,top=2.5cm,bottom=3cm,right=2.5cm]{geometry}
\newcommand{\ints}{\mathbb Z}
\newcommand{\rationals}{\mathbb Q}

\newcommand{\str}{\mathcal{O}}

\newcommand{\proj}{\mathbb{P}}

\newcommand{\complex}{\mathbb C}

\newcommand{\s}{\varepsilon}
\theoremstyle{plain}
\numberwithin{equation}{section}
\newtheorem{theorem}{Theorem}[section]
\newtheorem*{theorem*}{Theorem}
\newtheorem{proposition}[theorem]{Proposition}

\newtheorem{definition}[theorem]{Definition}
\theoremstyle{definition}

\newtheorem{remark}[theorem]{Remark}

\newtheorem{discussion}[theorem]{Discussion}

\begin{document}
\title{Seshadri constants on hyperelliptic surfaces}
\author[Krishna Hanumanthu]{Krishna Hanumanthu}
\address{Chennai Mathematical Institute, H1 SIPCOT IT Park, Siruseri, Kelambakkam 603103, India}
\email{krishna@cmi.ac.in}

\author[Praveen Kumar Roy]{Praveen Kumar Roy}
\address{Chennai Mathematical Institute, H1 SIPCOT IT Park, Siruseri, Kelambakkam 603103, India}
\email{praveenkroy@cmi.ac.in}

\thanks{Authors were partially supported by a grant from Infosys Foundation}

\date{January 30, 2018}
\maketitle
\begin{abstract}
We prove new results on single point Seshadri constants for ample line
bundles on hyperelliptic surfaces, motivated by the results in \cite{Far}. 
Given a hyperelliptic surface $X$ and an ample line bundle $L$ on $X$, 
we show that the least Seshadri constant $\s(L)$ of $L$ is a rational number
when $X$ is not of type 6. We also prove new lower bounds for the Seshadri
constant $\s(L,1)$ of $L$ at a very general point. 
\end{abstract}

\section{Introduction}\label{intro}

Let $X$ be a smooth complex projective variety and let $L$ be a line bundle on
$X$. The Seshadri criterion \cite[Theorem 7.1]{Har70} for ampleness says that $L$ is ample
if and only if there exists a real number $\varepsilon > 0$ such that
$L\cdot C \ge \varepsilon \cdot {\rm mult}_x C$, where $x \in X$ is an
arbitrary point and $C \subset X$ is any irreducible and
reduced curve containing $x$ (here ${\rm mult}_x C$
denotes the multiplicity of
the curve $C$ at $x$).  In other words, $L$ is ample if and only if
the infimum of the ratios $\frac{L\cdot C}{{\rm mult}_{x}C}$, over all
points $x$ and all irreducible and reduced curves $C$ passing through
$x$, is positive. 
Using this idea Demailly \cite{Dem} defined the notion of {\it
  Seshadri constants}. Given $X,L$ as above, 
the {\it Seshadri constant} of $L$ at $x \in X$ is defined as: 
$$\varepsilon(X,L,x):=  \inf\limits_{\substack{x \in C}} \frac{L\cdot
  C}{{\rm mult}_{x}C},$$ where the infimum is taken over all
irreducible and reduced curves passing through $x$. The Seshadri
criterion for ampleness can now be stated simply as follows: $L$ is ample if
and only if $\varepsilon(X,L,x) > 0$ for all $x \in X$. 

There are several interesting directions in which Seshadri
constants are being studied. See \cite{primer} for a comprehensive 
survey. 
One of the important problems in the study of Seshadri constants 
is computing them or bounding them. In the present article we focus on
this problem for hyperelliptic surfaces.
In general, Seshadri constants are difficult to compute precisely and
a lot of research has focussed on finding good lower and upper bounds.

Let $X$ be a smooth complex projective surface and let $L$ be an ample
line bundle on $X$. It is easy to see that $\varepsilon(X,L,x) \le
\sqrt{L^2}$ for all $x$. One then defines 
$$\varepsilon(X,L,1) : = \sup\limits_{x\in X}
\varepsilon(X,L,x).$$ 

It is known that $\varepsilon(X,L,1)$ is attained at a {\it very 
general} point $x \in X$; see
\cite{Ogu}. This means that $\varepsilon(X,L,1) = 
\varepsilon(X,L,x)$ for all $x$ 
outside a countable union of proper Zariski closed sets in 
$X$.

It is also known that if $\s(X,L,x) < \sqrt{L^2}$, then $\s(X,L,x) =  \frac{L\cdot
  C}{{\rm mult}_{x}C}$ for some curve $C$ (\cite[Proposition
1.1]{BS1}). If $\s(X,L,x) < \sqrt{L^2}$, then we say that it is {\it
  sub-maximal}. 
So sub-maximal Seshadri
constants are always rational, while a maximal Seshadri constant is
irrational if $L^2$ is not a square. However, no example is known of a triple
$(X,L,x)$ for which $\s(X,L,x) \notin \rationals$. 

At the other end of the interval, one defines 
$$\s(X,L) := \inf\limits_{\substack{x \in X}}  \s(X,L,x).$$ 

It is easy
to see that $\s(X,L) > 0$ for ample $L$. In fact, $\s(X,L) \ge
\frac{1}{n}$, if $nL$ is very ample. Just like $\s(X,L,1)$, it is known that $\s(X,L) =
\s(X,L,x)$ for some $x \in X$ (see \cite{Ogu}). But unlike $\s(X,L,1)$, which
is attained at very general points, $\s(X,L)$ is attained at {\it special}
points. In general, one has the following inequalities for any point
$x \in X$: 
$$0 < \s(X,L) \le \s(X,L,x) \le \s(X,L,1) \le \sqrt{L^2}.$$ 
Further, it follows from the previous paragraph that $\s(X,L) \in
\rationals$, except when $L^2$ is not a square and 
$\s(X,L,x) = \s(X,L) = \s(X,L,1) = \sqrt{L^2}$ for all $x \in X$.

The above discussion leads to an interesting dynamic in the study of
bounds on Seshadri constants. In many situations, the Seshadri
constants of $L$  at very general points may be expected
to be maximal, i.e., equal to $\sqrt{L^2}$.
On the other
hand, the Seshadri constants at special points (and $\s(X,L)$) are expected to be
sub-maximal and hence rational.  This leads to two contrasting
problems. On the one hand, the focus has been to find good lower
bounds for $\s(X,L,1)$ which are very close to the maximal value
$\sqrt{L^2}$. The second problem is to calculate the Seshadri
constants at special points and try to prove that $\s(X,L) \in
\rationals$, for instance. These are very different problems because the
first one uses information about curves passing through very general
points on the surface, while the second problem requires knowledge of
specific curves passing through special points on the surface. See
also Discussion \ref{type6-special}. 

For a sampling of the many results in this area, see 
\cite{EL,Bau,BS,KSS,BS1,Far}. For a detailed account, see \cite{primer}.

In this article, we address both the problems discussed above in the case of
hyperelliptic surfaces. Our primary motivation is \cite{Far}, where
several results on Seshadri constants on hyperelliptic surfaces are proved.  

Hyperelliptic surfaces are minimal surfaces of Kodaira dimension 0 and
irregularity 1. They are realized as finite group quotients of 
products of two elliptic curves. These surfaces have been classified
and are known to belong to one of seven different types. They all 
have Picard rank 2 and the free group 
Num$(X)$ of divisors modulo numerical equivalence is well-understood. 
See Section \ref{prelims} for more details. 

Let $X$ be a hyperelliptic surface and $L$ be an ample line bundle on
$X$. We first consider the problem of computing $\s(X,L)$ in 
subsection \ref{special}.  In our main result Theorem \ref{rational} in
this subsection, we show that $\s(X,L) \in \rationals$
provided $X$ is not of type 6. This partially 
answers \cite[Question 1.6]{SS}, which asks if $\s(X,L)$ is always
rational for any pair $(X,L)$. Some affirmative answers to this
question are known (\cite{Bau,BS,S,Syz,Fue}), but it is open in general. 
In other results in this subsection, we also explicitly compute
$\s(X,L)$ in some cases. 

In subsection \ref{general}, we study the Seshadri constant of $L$ at
a very general point $x$. One of our main results, 
Theorem \ref{general-seshadri-constant}, says that $\s(X,L,1) \ge
(0.93)\sqrt{L^2}$, or $\s(X,L,1)$ is equal to one of two easily computable
natural numbers. 
Let $L$ be of numerical type $(a,b)$. Then
depending on how $a$ and $b$ are related to each other, we either explicitly
compute $\s(X,L,1)$ or show that $\s(X,L,1) \ge (0.93) \sqrt{L^2}$. We
have such a result for each of the seven types of hyperelliptic
surfaces. There are several results in the literature giving lower bounds for
$\s(X,L,1)$ when $X$ is an arbitrary surface and $L$ is any ample
line bundle.  In Remark \ref{compare}, we compare our bound 
$(0.93) \sqrt{L^2}$ with some existing bounds and
note that it is often better.

We work over $\complex$, the field
of complex numbers.  A {\it surface} is a
two-dimensional smooth complex projective variety. When the surface
$X$ is clear from the context, we denote Seshadri constants simply by
$\s(L,x), \s(L)$, or $\s(L,1)$. 

\section{Preliminaries}\label{prelims}


\begin{definition} 
A hyperelliptic surface $X$ is a minimal smooth 
surface with Kodaira dimension $\kappa(X)=0$ satisfying  
$h^1(X,\str_X) = 1$ and $h^2(X,\str_X) = 0$.

\end{definition}

Hyperelliptic surfaces are also known as {\it bielliptic surfaces}
(cf. \cite{Bea,Se}). 
We recall below some key properties of hyperelliptic surfaces that we use
repeatedly. More details can be found in \cite{Bea,Se}. We follow the
notation in \cite{Se,Far}.

There is an alternate characterization of hyperelliptic
surfaces. A smooth surface $X$ is hyperelliptic if and only if $X \cong (A\times B)/G$, where $A$ and $B$ are
elliptic curves and $G$ is a finite group 
of translation of $A$ acting on $B$ in such a way that $B/G \cong
\proj^1$.

We have the following diagram:
$$
\xymatrix{
X \cong (A\times B)/G \ar[d]_{\Psi}\ar[r]^-{\Phi}  & A/G \\
B/G \cong \proj^1
}
$$

In the above diagram $\Phi$ and $\Psi$ are natural projections.
The fibres of $\Phi$ are all smooth and isomorphic to $B$. 
The fibres of $\Psi$ are all multiples of smooth elliptic curves, and
all but finitely many of them are smooth and isomorphic to $A$. 

Hyperelliptic surfaces were classified more than hundred years ago 
by G. Bagnera and M. de Franchis by analyzing the group $G$ and its action on $B$. 
They showed that every hyperelliptic surface is of one of the seven types
listed in the table below;
see \cite[V1.20]{Bea}.

Every hyperelliptic surface has Picard rank 2. Serrano \cite{Se} has
described a basis for the free group Num($X$) of divisors modulo
numerical equivalence for each of the seven types of hyperelliptic
surfaces. For each type, Serrano also lists the multiplicities
$m_1,\ldots,m_s $ of the
singular fibres of $\Psi$, where $s$ is the number of singular fibres.

\begin{theorem} \label{serrano} \cite[Theorem 1.4]{Se}. 
Let $X \cong (A\times B)/G $ be a hyperelliptic surface. 
A basis for the group 
Num($X$) of divisors modulo
numerical equivalence and the multiplicities of the singular  fibres of $\Psi: X
\to B/G$ in each type are given in the following table.
\begin{center} 

  \begin{tabular}{|c| c| c| c|}
 \hline
 
 Type of $X$  & $G$ &$ m_1,m_2,
\ldots,m_s$ & Basis of Num($X$)\\
 \hline 
 1 & $\mathbb{Z}_{2}$ &$2,2,2,2$&$A/2$, $B$ \\ 
 2 & $\mathbb{Z}_{2} \times \mathbb{Z}_{2}$&$2,2,2,2$&$A/2$, $B/2$ \\
 3 & $\mathbb{Z}_{4}$ &$2,4,4$&$A/4$, $B$ \\
 4 & $\mathbb{Z}_{4} \times \mathbb{Z}_{2}$ & $2,4,4$ &$A/4$,$B/2$  \\
 5 & $\mathbb{Z}_{3}$ & $3,3,3$ & $A/3, B$ \\
 6 & $\mathbb{Z}_{3} \times \mathbb{Z}_{3}$ & $3,3,3$ & $A/3,B/3$ \\
 7 & $\mathbb{Z}_{6}$ & $2,3,6$ & $A/6,B$ \\
\hline 
  
 \end{tabular}

\end{center}
\end{theorem}

Let $X$ be a hyperelliptic surface. 
Let $\mu = lcm(m_1,m_2,\ldots,m_s)$ and let $\gamma = |G|$. By Serrano's theorem, a basis of Num($X$)
is given by $A/\mu$, $(\mu/\gamma)B$. 

\noindent {\bf Notation: }We say that {\it $L$ is a line bundle of type $(a,b)$} on
$X$ if $L$ is numerically equivalent to  $a\cdot A/\mu +
b\cdot (\mu/\gamma)B$. If $L$ is of type $(a,b)$, we write $L \equiv
(a,b)$. 

We note the following properties of line bundles on $X$. 

\begin{enumerate}
\item $A^2 = 0 , B^2 = 0 , A\cdot B = \gamma$.
\item A divisor $b\cdot (\mu/\gamma)B \equiv  (0,b)$ is effective if and only if 
$b(\mu/\gamma) \in \mathbb{N}$
(\cite[Proposition 5.2]{Apr}). 
\item A line bundle of type $(a,b)$ is ample if and only if $a > 0$
    and $b > 0$ (\cite[Lemma  1.3]{Se}).
\item If $C$ is an irreducible and reduced curve on $X$ and $x \in C$
  is a point of multiplicity $m$, then $C^2 \ge m^2-m$. 
\end{enumerate}

The inequality in (4) follows from the genus formula, and the facts
that the canonical divisor is numerically trivial on a hyperelliptic
surface and that there are no rational curves on a
hyperelliptic surface. 

We also use the following important lower bound on self-intersection of a curve
$C$ passing through a very general point. See \cite{EL,X1,KSS,Bas}, for instance.

\begin{theorem}\label{xu}
Let $X$ be a hyperelliptic surface and let $C$ be an irreducible and
reduced curve on $X$. Suppose that $C$ passes through a very general
point $x \in X$ with multiplicity $m \ge 2$. Then $C^2 \ge m^2-m+2$. 
\end{theorem}

\section{Seshadri constants}\label{results}
In this section we first consider $\s(L)$ and then prove our results
on $\s(L,1)$.
\subsection{Results about $\s(L)$.}\label{special}
\begin{theorem}\label{odd}
Let $X$ be a hyperelliptic surface of odd type (i.e., of type 1, 3, 5,
or 7). Let $L \equiv (a,b)$ be an
ample line bundle on $X$. 
Then $\s(L) = \min\{a,b\}$. 
\end{theorem}

\begin{proof}
We first prove that $\s(L,x) \ge \min\{a,b\}$ for any $x \in X$. We
then show that equality holds for a suitable $x$. 

Note that since $X$ is a
hyperelliptic surface of odd type, $\mu = \gamma$. Hence $B$ is given
by $(0,1)$ in Num$(X)$. On the other hand, $A$ is given by $(2,0),
(4,0), (3,0)$ and $(6,0)$ in types 1, 3, 5 and 7, respectively.

Since the fibres of $\Phi: X \to A/G$ cover $X$, are smooth and are
isomorphic to $B$, there is a smooth curve which is numerically equivalent to
$(0,1)$ that contains any given point $x$. 
Similarly, the fibres of $\Psi: X \to B/G$ cover $X$, but they are not
all smooth. The smooth fibres of $\Psi$ are isomorphic to $A$ and singular
fibres are multiples of smooth fibres. The number of singular fibres
and their multiplicities are completely determined by the type of
$X$. See the table in Theorem \ref{serrano}.

Now let $x \in X$ be an arbitrary point. Let $C$ be a reduced
and irreducible curve on $X$ passing through $x$ with multiplicity
$m \ge 1$. We consider three possibilities for $C$. First, it is a
fibre of $\Phi$; second, it is a fibre of $\Psi$; and third, it is
different from the fibres of $\Phi$ and $\Psi$. 

If $C$ is a fibre of $\Phi$, then $C$ is smooth and is isomorphic to $B$ and is
numerically equivalent to $(0,1)$. In this case, $m = 1$. So the
Seshadri ratio is $L\cdot C = a$. 

If $C$ is a fibre of $\Psi$, then $C$ is not necessarily smooth.
Numerically, $C$ is given by $(\mu,0)$. 
The multiplicity $m$ is determined by the table in Theorem \ref{serrano}. For instance, if $X$
has type 1, then $m = 1$, or $2$. Or, if $X$ has type 3, then $m = 1,
2$, or $4$. In this case, the
Seshadri ratio is $\frac{L\cdot C}{m} = \frac{\mu b}{m}$. 

Now let $C$ be different from the fibres of $\Psi$ and $\Phi$. Let $C$ be
represented by $(\alpha,\beta)$ in Num$(X)$. We use Bezout's theorem
to bound the values of $\alpha$ and $\beta$. Since $x$ is a point of a
smooth fibre $(0,1)$, we have $C \cdot (0,1) = \alpha \ge m$. 
On the other hand, the fibre of $\Psi$ containing $x$ may not be
smooth. In this case, Bezout's theorem gives 
$C \cdot (\mu, 0 ) = \mu \beta \ge mn$, where $n$ is 
the multiplicity of 
the fibre of $\Psi$ containing $x$. 
Thus we have $\frac{L\cdot C}{m} = \frac{a\beta+b\alpha}{m} \ge
b+\frac{an}{\mu}$. 

Since $\mu \ge m$ and $n\ge 1$, we conclude that the Seshadri ratio 
$\frac{L\cdot C}{m} \ge \min(a,b,b+\frac{a}{\mu})$ for any reduced irreducible
curve $C$ passing through $x$. Hence 
$\s(L,x) \ge \min(a,b,b+\frac{a}{\mu})
\ge \min(a,b)$.

Now let $x$ be a point on a singular fibre of $\Psi$ with the maximum possible 
multiplicity. For instance, if $X$ has type 7, $x$ is any point on a fibre of
$\Psi$ of multiplicity 6.
Then, in the notation above, $m = n = \mu$. 
So $\s(L,x) = \min(a,b,a+b) = \min(a,b)$. This completes the proof of the
theorem. 
\end{proof}

\begin{remark} 
The result in Theorem \ref{odd} is proved for hyperelliptic surfaces
of type 1 in \cite[Theorem 3.4]{Far} and our proof essentially follows
from the arguments given by Farnik. In fact, Farnik gives a precise
value for $\s(L,x)$ for any $x$ and any ample line bundle $L$ on a
hyperelliptic surface of type 1.
\end{remark}

Our next result partially answers \cite[Question 1.6]{SS} for hyperelliptic
surfaces. This question asks if $\s(X,L)$ is rational for any surface $X$ and
any ample line bundle $L$ on $X$. 
 So far an affirmative answer to this question has been found
in some cases. 

The case of quartic surfaces $X \subset \proj^3$ and $L = \str_X(1)$ is considered in
\cite[Theorem]{Bau}. It is proved that $\s(X,L)=1,4/3$ or $2$, depending
on certain geometric properties of $X$. In \cite[Theorem A.1]{BS},
it is proved that $\s(X,L)$ is rational if $X$ is an abelian surface
and $L$ is any ample line bundle on $X$. The same result is shown for
Enriques surfaces in \cite[Theorem 3.3]{S}. Finally, \cite{Syz,Fue}
study minimal ruled surfaces. Such surfaces are geometrically ruled
over a smooth curve $C$ and one attaches a certain invariant $e \in
\ints$ to them. If $e \ge 0$, then \cite[Theorem 3.27]{Syz} and \cite[Theorem
4.14]{Fue} show that $\s(X,L) \in \rationals$ for any ample line bundle
$L$ on $X$.

\begin{theorem}\label{rational}
Let $X$ be a hyperelliptic surface of type different from 6 and let $L$ be an ample line bundle
on $X$. Then $\s(L)$ is rational. 
\end{theorem}

\begin{proof}
Let $X$ be a hyperelliptic surface of type different from 6 and let
$L\equiv (a,b)$ be an ample line bundle
on $X$.  If $X$ has odd type then the assertion follows from Theorem \ref{odd}. 

\underline{$X$ is of type 2:}
If $2a=b$, then $L^2 = 2ab$ is a perfect square and $\s(L)$ is a
rational number (for instance, see \cite[Corollary 1.8]{SS}). 

If $b < 2a$, let $x$ be a point on a singular fibre of $\Psi$. This
fibre is numerically equivalent to $(2,0)$ and $x$ is a point of
multiplicity 2 on it. So $\s(L,x) \le \frac{(a,b)\cdot (2,0)}{2} = b <
\sqrt{2ab} = \sqrt{L^2}$. It follows by \cite[Lemma 1.7]{SS} that $\s(L)
\in \rationals$. 
On the other hand, if $b > 2a$, then let $x
\in X$ be any point. Then $x$ belongs to a fibre of $\Phi$. Note that
all the fibres of $\Phi$ are smooth and are numerically equivalent to
$(0,2)$. So $\s(L,x) \le \frac{(a,b)\cdot (0,2)}{1} = 2a <
\sqrt{L^2}$. Again it follows that $\s(L) \in \rationals$. 
Note that in fact $\s(L,1) \le 2a < \sqrt{L^2}$, if $b > 2a$.

\underline{$X$ is of type 4:} As in the above case, if $2a=b$, then
$\s(L) \in \rationals$. 

If $b < 2a$, let $x$ be a point on a fibre of $\Psi$ of multiplicity
4. Then $\s(L,x) \le  b < \sqrt{L^2}$. On the other hand, let $b > 2a$
and let $x$ be any point. Consider the fibre of $\Phi$ containing
$x$. This fibre is smooth and numerically equivalent to $(0,2)$. Again
as before, $\s(L,x) \le 2a < \sqrt{L^2}$. 
\end{proof}

We have the following result for type 6 hyperelliptic surfaces. 

\begin{theorem}\label{type6-rational}
Let $X$ be a hyperelliptic surface of type 6 and let $L \equiv (a,b)$ be an
ample line bundle on $X$ such that $b$ is not in the interval
$(2a,9a/2)$. Then $\s(L) \in \rationals$.
\end{theorem}
\begin{proof}
If $b=2a$ or $b=9a/2$, then $L^2 = 2ab$ is a square and $\s(L)\in
\rationals$. So we assume that either $b < 2a$ or $b > 9a/2$. 

If $b< 2a$, choose a point $x$ on a singular fibre of $\Psi$. Then the
fibre is represented numerically by $(3,0)$ and the multiplicity of
the fibre at $x$ is 3. So $\s(L,x) \le \frac{(a,b)\cdot (3,0)}{3} = b <
\sqrt{2ab}.$ If $b > 9a/2$, then choose any point $x$ and consider a
fibre of $\Phi$ containing it. We have $\s(L,x) \le \frac{(a,b)\cdot (0,3)}{1} = 3a <
\sqrt{2ab}$. Thus $\s(L) \in \rationals$.
\end{proof}

\begin{discussion}\label{type6-special}
Let $X$ be any surface, and let $L$ be ample on $X$. If $x \in X$,
then an easy upper bound for $\s(L,x)$ is given by $\frac{L\cdot C}{{\rm
  mult}_x C}$, where $C$ is a curve containing $x$, {\it provided} this ratio is smaller than
$\sqrt{L^2}$. 

Of course, there are no such curves if $\s(X,L) = \sqrt{L^2}$.
We note below why there are no obvious examples of such curves 
if $X$ is a hyperelliptic surface of type 6 and $L\equiv (a,b)$ with $b \in (2a,9a/2)$. 

According to Theorems \ref{rational} and \ref{type6-rational},
if $X$ is a hyperelliptic surface of type different from 6, or if $X$
has type 6, but $L \equiv  (a,b)$ with $b \notin (2a,9a/2)$, 
then $\frac{L\cdot C}{{\rm  mult}_x C} < \sqrt{L^2}$, for a suitable
$x$ and a suitable fibre $C$ of $\Psi$ or $\Phi$. 
This in turn allows us to conclude $\s(X,L) \in \rationals$ in 
these cases. It is also clear from the proof of Theorem \ref{type6-rational}
that if $X$ has type 6 and $L \equiv (a,b)$ with $b \in (2a,9a/2)$, then 
$\frac{L\cdot C}{{\rm mult}_x C} \ge \sqrt{L^2}$, for {\it any} fibre $C$ of
$\Psi$ or $\Phi$.

In general, for a surface $X$ and an ample line bundle $L$ on $X$, 
in order to conclude that $\s(X,L) \in \rationals$, we must establish the
existence of a suitable pair $x \in C$ for which 
$\frac{L\cdot C}{{\rm mult}_x C} <\sqrt{L^2}$.  
If $X$ is a hyperelliptic surface of type
6 and $L\equiv (a,b)$ with $b \in (2a,9a/2)$, there are no obvious 
candidates for such a pair. One needs more specific information about
singular curves on such a surface. 
\end{discussion}

We do however give a lower bound for $\s(L,x)$
for any $x$ in the following proposition.

\begin{proposition}\label{type6-bound}
Let $X$ be a hyperelliptic surface of type
6 and let $L\equiv (a,b)$ be an ample line bundle with $b \in
(2a,9a/2)$. Then $\s(L,x)\ge (0.7)\sqrt{L^2}$ for all $x \in X$. 
\end{proposition}
\begin{proof}
If $\s(L,x) < (0.7)\sqrt{L^2}$ for some $x \in X$, then $\s(L,x) = \frac{L\cdot C}{{\rm
    mult}_x C}$ for an irreducible and reduced curve $C \equiv  (\alpha,
\beta)$ containing $x$. Let $m = {\rm  mult}_x C$. If $m=1$, then
$L\cdot C < \sqrt{L^2}$. Then the Hodge Index Theorem gives $L^2 . C^2
\le (L\cdot C)^2 < L^2$. So $C^2 = 2\alpha\beta < 1$. Thus $\alpha=0$
or $\beta = 0$. Then $C$ is a fibre of $\Phi$ or $\Psi$. But this
is not possible, as mentioned in Discussion \ref{type6-special}.

So assume $m \ge 2$. We know $C^2 \ge m^2-m$ (see Section \ref{prelims}). Applying the Hodge Index
Theorem again, we get $m^2-m < (0.7)^2m^2$, which gives $(0.51)m^2 - m
< 0$. But this is not possible for $m\ge 2$. 
\end{proof}

We use the idea in the above proof again in Theorem 
\ref{general-non-fibre}.

\begin{proposition}\label{even-1}
Let $X$ be a hyperelliptic surface of even type. Let $L \equiv  (a,b)$ be an
ample line bundle on $X$ satisfying the following:
\begin{enumerate}
\item $b \le a$ if $X$ is of type 2;
\item $2b \le a$ if $X$ is of type 4 or 6.  
\end{enumerate}

Then $\s(L) = b$.  
\end{proposition}
\begin{proof}
First let $X$ be of type 2. If $x$ is a point on a singular
fibre of $\Psi$, then as in the proof of Theorem \ref{rational},
$\s(L,x) \le b$. 

Now let $x \in X$ be an arbitrary point. Then $x$ is in a fibre of
$\Phi$ which is represented by $(0,2)$. The Seshadri ratio for this
fibre is $\frac{L\cdot (0,2)}{1} = 2a \ge b$. The Seshadri ratio for
any fibre of $\Psi$ containing $x$ is at least $b$. 
On the other hand, let 
$C \equiv  (\alpha,\beta)$ be an irreducible and reduced curve, different
from the fibres of $\Psi$ or
$\Phi$, passing through $x$ with multiplicity $m \ge 1$. 
Then (as in the
proof of Theorem \ref{odd}) Bezout's
theorem gives $2\alpha\ge m$ and $2\beta \ge m$. So 
$\frac{L\cdot C }{m} = \frac{a\beta+b\alpha}{m} \ge \frac{a+b}{2} \ge
b$. In other words, $\s(X,L,x) = \inf \frac{L \cdot C}{{\rm mult}_x C}
\ge b$. 

Thus $\s(L,x) \ge b$ for all $x \in X$ and $\s(L,x) \le b$ if $x$ is
on a singular fibre of $\Psi$. It follows that $\s(L) = b$. 

Now let $X$ be of type 4 or 6. As in the above case, if $x$ is on a
singular fibre of $\Psi$, then $\s(L,x) \le b$ (when $X$ is of type 4,
  we take the point $x$ on a fibre of multiplicity 4). 

Now let $x \in X$ be arbitrary and let $C \equiv (\alpha,\beta)$ 
be an irreducible and reduced curve, different
from the fibres of $\Psi$ or $\Phi$, passing through $x$ with 
multiplicity $m \ge 1$. Then we have $4\beta \ge m$ and $2\alpha \ge
m$ when $X$ is of type 4 and 
$3\beta \ge m$ and $3\alpha \ge
m$ when $X$ is of type 6. In either case,
$\frac{L\cdot C }{m} = \frac{a\beta+b\alpha}{m} \ge b$. As above, we
conclude that $\s(L) = b$. 
\end{proof}

\begin{proposition}\label{even-2}
Let $X$ be a hyperelliptic surface of even type. Let $L \equiv  (a,b)$ be an
ample line bundle on $X$. Then the following statements hold: 
\begin{enumerate}
\item Let $X$ be of type 2. If $b \ge 3a$, then $\s(L,x) = 2a$ for all
  $x \in X$. 
\item Let $X$ be of type 4. If $b \ge 7a/2$, then $\s(L,x) = 2a$ for all
  $x \in X$. 
\item Let $X$ be of type 6. If $b \ge 8a$, then $\s(L,x) = 3a$ for all
  $x \in X$. 
\end{enumerate}
\end{proposition}
\begin{proof}
The proof is similar to the proof of Proposition \ref{even-1}, so we will
only give a brief sketch. 

First let $X$ be of type 2. Let $x \in X$ be any point. Since a fibre
of $\Phi$ contains $x$, we have $\s(L,x) \le \frac{L\cdot (0,2)}{1} =
2a$. If $x$ is on a singular fibre of $\Psi$, then the corresponding
Seshadri ratio is $\frac{L\cdot (2,0)}{2} = b \ge 2a$. If $x$ is on a
smooth fibre of $\Psi$, then the corresponding
Seshadri ratio is $\frac{L\cdot (2,0)}{1} = 2b \ge 2a$.

Now let $C\equiv  (\alpha, \beta)$ be an irreducible and reduced curve, different
from the fibres of $\Psi$ or $\Phi$, passing through $x$ with multiplicity $m \ge 1$. 
Bezout's theorem gives $2\alpha \ge m$ and $2\beta \ge m$. So 
$\frac{L\cdot C }{m} = \frac{a\beta+b\alpha}{m} \ge \frac{a+b}{2} \ge
2a$, by hypothesis.  So we conclude that $\s(L,x) = 2a$ for all $x\in X$.

The proof for types 4 and 6 is similar.
\end{proof}

\subsection{Results about $\s(L,1)$.}\label{general}

\begin{theorem}\label{general-non-fibre}
Let $X$ be a hyperelliptic surface and let $L$ be an ample line bundle
on $X$. Suppose that $C \equiv  (\alpha, \beta)$ is an irreducible, reduced 
curve with $\alpha \ne 0$,
$\beta \ne 0$ and which passes through a very general point with
multiplicity $m \ge 1$. Then $\frac{L\cdot C}{m} \ge (0.93)\sqrt{L^2}$. 
\end{theorem}
\begin{proof}
First, let $m=1$. If $L\cdot C < (0.93)\sqrt{L^2}$, then the
Hodge Index Theorem gives $C^2 < (0.93)^2$. So $C^2 = 2\alpha \beta =
0$, which violates the hypothesis on $C$. 

So assume $m \ge 2$. Then we have $C^2 \ge m^2-m+2$, by 
Theorem \ref{xu}. Again, the Hodge
Index Theorem gives $m^2-m+2 < (0.93)^2m^2$. So $m$ satisfies the
quadratic relation $(0.13)m^2-m+2 < 0$. But it is easy to see that the
quadratic expression $(0.13)m^2-m+2$ is always positive, since it
grows as $m$ goes to $\infty$ or $-\infty$ and its discriminant is
$1-8(0.13) = -0.04 < 0$. 
\end{proof}
\begin{remark}
In the above proof, we used the 
fact that the quadratic  $(1-\delta^2)m^2-m+2$ is positive for all $m \ge 1$, 
where $\delta = 0.93$. In order to get a better bound in Theorem 
\ref{general-non-fibre}, we have to increase $\delta$. But this
  forces the above quadratic to become negative for some $m$. For instance, if
  $\delta = 0.94$, then the quadratic $(1-0.94^2)m^2-m+2$ is negative
  for $m=4, 5$. Similarly, for $\delta=0.99$, the quadratic 
$(1-0.99^2)m^2-m+2$ is negative for $2 \le m \le 48$. 
As $\delta$ approaches 1, the set $\{m ~|~ (1-\delta^2)m^2-m+2 <
0\}$ keeps increasing.


As $\delta$ approaches 1, more precise information about
$L\cdot C$ for curves passing through very general points will be
required to prove the inequality  $\frac{L\cdot C}{m} \ge \delta
\sqrt{L^2}$. This may be possible to do for specific line bundles $L$. 
\end{remark}

As a corollary to Theorem \ref{general-non-fibre},  we obtain our main theorem about $\s(L,1)$ for
ample line bundles on hyperelliptic surfaces. 

\begin{theorem}\label{general-seshadri-constant}
Let $X$ be a hyperelliptic surface and let $L$ be an ample line bundle
on $X$. If $\s(L,1) < (0.93)\sqrt{L^2}$, then $\s(L,1) =
\min(L\cdot A, L\cdot B)$.
\end{theorem}
\begin{proof}
If $\s(L,1) \ge (0.93) \sqrt{L^2}$, then there is nothing to prove. Otherwise,
we have $\s(L,1) = \frac{L\cdot C}{m}$, where $C$ is a reduced and
irreducible curve passing through a very general point with
multiplicity $m$. Let $C \equiv (\alpha,\beta)$. By Theorem
\ref{general-non-fibre}, either $\alpha = 0$ or $\beta = 0$.  
In other words, $C$ is a fibre of $\Phi$ or $\Psi$. Since $x$ is a
very general point, we may assume that it does not lie on any of the 
finitely many singular fibres of $\Psi$. Thus $C$ is smooth and
isomorphic to $B$ or $A$. Hence $\s(L,1) = \min(L\cdot A, L\cdot B)$. 
\end{proof}

We next consider different types of hyperelliptic surfaces and prove
specific results about $\s(L,1)$. 

\begin{theorem}\label{type1}
Let $X$ be a hyperelliptic surface of type 1. Let $L \equiv (a,b)$ be an
ample line bundle on $X$. Then $\s(L,1) = \min(a,2b)$.
\end{theorem}
\begin{proof}
We repeat the proof that is already essentially given in 
\cite[Theorem 3.4]{Far}. This proof illustrates the special property
of type 1 hyperelliptic surfaces in the sense that the Seshadri
constants are always computed by the fibres of $\Phi$ or $\Psi$. 
 
Note that when $X$ has type 1, a fibre $B$ of $\Phi$ is given by
$(0,1)$ and a smooth fibre $A$ of $\Psi$ is given by $(2,0)$. So
$L\cdot A = 2b$ and $L \cdot B = a$. Since a very general point $x
\in X$ always belongs to a fibre $B$ and a smooth fibre $A$, we have 
$\s(L,1) = \s(L,x) \le \min(a,2b)$. 

Now if $C \equiv  (\alpha, \beta)$ is a curve with $\alpha\beta \ne 0$ and
it passes through a very general point with multiplicity $m$, then
Bezout's theorem gives $\alpha \ge m$ and $\beta \ge m/2$. Thus
$\frac{L\cdot C}{m} = \frac{a\beta+b\alpha}{m} \ge \frac{a}{2}+b \ge
\min(a,2b)$. It follows that $\s(L,1) = \min(a,2b)$.
\end{proof}


\begin{theorem}\label{type2}
Let $X$ be a hyperelliptic surface of type 2. Let $L \equiv (a,b)$ be an
ample line bundle on $X$. Then the following statements hold: 
\begin{enumerate}
\item  If $\frac{2\min(a,b)}{(0.93)^2} \le
\rm{max}(a,b)$, then $\s(L,1) = 2\min(a,b)$. 
\item If $\frac{2\min(a,b)}{(0.93)^2} \ge \rm{max}(a,b)$, then $\s(L,1) \ge
(0.93)\sqrt{L^2}$.
\end{enumerate}
\end{theorem}
\begin{proof}
Note that when $X$ has type 2, a fibre $B$ of $\Phi$ is given by
$(0,2)$ and a smooth fibre $A$ of $\Psi$ is given by $(2,0)$. So
$L\cdot A = 2b$ and $L \cdot B = 2a$. Since a very general point $x
\in X$ always belongs to a fibre $B$ and a smooth fibre $A$, we have 
$\s(L,1) = \s(L,x) \le \min(2a,2b)$. Also, by Theorem
\ref{general-seshadri-constant}, $\s(L,1) \ge (0.93)\sqrt{L^2}$ or
$\s(L,1) = \min(2a,2b)$. Note that $L^2 = 2ab$. 

We have
\begin{eqnarray*}
\frac{2\min(a,b)}{(0.93)^2} &\le& \rm{max}(a,b) \\
\Leftrightarrow  4(\min(a,b))^2 &\le& (0.93)^2(2ab) \\
\Leftrightarrow 2\min(a,b) &\le&  (0.93)\sqrt{2ab} \\
\Rightarrow\s(L,1) &=&  \min(2a,2b).
\end{eqnarray*}

On the other hand, 

$\frac{2\min(a,b)}{(0.93)^2} \ge \rm{max}(a,b) \Leftrightarrow 
2\min(a,b) \ge (0.93)\sqrt{2ab} \Rightarrow$  
$\s(L,1) \ge (0.93)\sqrt{2ab}$.
\end{proof}

\begin{theorem}\label{type3}
Let $X$ be a hyperelliptic surface of type 3. Let $L \equiv (a,b)$ be an
ample line bundle on $X$. Then the following statements hold: 
\begin{enumerate}
\item  If $b \le \frac{a(0.93)^2}{8}$, then $\s(L,1) = 4b$. 
\item If $\frac{a(0.93)^2}{8} \le b \le \frac{a}{2(0.93)^2}$,
then $\s(L,1) \ge (0.93)\sqrt{L^2}$.
\item If $b \ge \frac{a}{2(0.93)^2}$, then $\s(L,1) = a$.
\end{enumerate}
\end{theorem}
\begin{proof}
Note that when $X$ has type 3, a fibre $B$ of $\Phi$ is given by
$(0,1)$ and a smooth fibre $A$ of $\Psi$ is given by $(4,0)$. So
$L\cdot A = 4b$ and $L \cdot B = a$. Since a very general point $x
\in X$ always belongs to a fibre $B$ and a smooth fibre $A$, we have 
$\s(L,1) = \s(L,x) \le \min(a,4b)$. Also, by Theorem
\ref{general-seshadri-constant}, $\s(L,1) \ge (0.93)\sqrt{L^2}$ or
$\s(L,1) = \min(a,4b)$.

If $b \le \frac{a(0.93)^2}{8}$, then clearly $4b \le a$. Further 
$b \le \frac{a(0.93)^2}{8} \Leftrightarrow  4b \le
(0.93)\sqrt{2ab}$. So $\s(L,1) = 4b$. 

If $b \ge \frac{a}{2(0.93)^2}$, then clearly $a \le 4b$. Further 
$b \ge \frac{a}{2(0.93)^2}\Leftrightarrow a \le (0.93)\sqrt{2ab}$. So
$\s(L,1) = a$. 

Finally, if $\frac{a(0.93)^2}{8} \le b \le \frac{a}{2(0.93)^2}$, then
$a \ge (0.93)\sqrt{2ab}$ and $4b \ge (0.93)\sqrt{2ab}$. So $\s(L,1)
\ge (0.93)\sqrt{2ab}$. 
\end{proof}

\begin{theorem}\label{type4}
Let $X$ be a hyperelliptic surface of type 4. Let $L \equiv (a,b)$ be an
ample line bundle on $X$. Then the following statements hold: 
\begin{enumerate}
\item  If $b \le \frac{a(0.93)^2}{8}$, then $\s(L,1) = 4b$. 
\item If $\frac{a(0.93)^2}{8} \le b \le \frac{2a}{(0.93)^2}$,
then $\s(L,1) \ge (0.93)\sqrt{L^2}$.
\item If $b \ge \frac{2a}{(0.93)^2}$, then $\s(L,1) = 2a$.
\end{enumerate}
\end{theorem}
\begin{proof}
Note that when $X$ has type 4, a fibre $B$ of $\Phi$ is given by
$(0,2)$ and a smooth fibre $A$ of $\Psi$ is given by $(4,0)$. So
$L\cdot A = 4b$ and $L \cdot B = 2a$. Since a very general point $x
\in X$ always belongs to a fibre $B$ and a smooth fibre $A$, we have 
$\s(L,1) = \s(L,x) \le \min(2a,4b)$. Also, by Theorem
\ref{general-seshadri-constant}, $\s(L,1) \ge (0.93)\sqrt{L^2}$ or
$\s(L,1) = \min(2a,4b)$. 

If $b \le \frac{a(0.93)^2}{8}$, then clearly $4b \le 2a$. Further 
$b \le \frac{a(0.93)^2}{8} \Leftrightarrow  4b \le
(0.93)\sqrt{2ab}$. So $\s(L,1) = 4b$. 

If $b \ge \frac{2a}{(0.93)^2}$, then clearly $2a \le 4b$. Further 
$b \ge \frac{2a}{(0.93)^2}\Leftrightarrow 2a \le (0.93)\sqrt{2ab}$. So
$\s(L,1) = 2a$. 

Finally, if $\frac{a(0.93)^2}{8} \le b \le \frac{2a}{(0.93)^2}$, then
$2a \ge (0.93)\sqrt{2ab}$ and $4b \ge (0.93)\sqrt{2ab}$. So $\s(L,1)
\ge (0.93)\sqrt{2ab}$. 
\end{proof}

\begin{theorem}\label{type5}
Let $X$ be a hyperelliptic surface of type 5. Let $L \equiv  (a,b)$ be an
ample line bundle on $X$. Then the following statements hold: 
\begin{enumerate}
\item  If $b \le \frac{2a(0.93)^2}{9}$, then $\s(L,1) = 3b$. 
\item If $\frac{2a(0.93)^2}{9} \le b \le \frac{a}{2(0.93)^2}$,
then $\s(L,1) \ge (0.93)\sqrt{L^2}$.
\item If $b \ge \frac{a}{2(0.93)^2}$, then $\s(L,1) = a$.
\end{enumerate}
\end{theorem}
\begin{proof}
Note that when $X$ has type 5, a fibre $B$ of $\Phi$ is given by
$(0,1)$ and a smooth fibre $A$ of $\Psi$ is given by $(3,0)$. So
$L\cdot A = 3b$ and $L \cdot B = a$. Since a very general point $x
\in X$ always belongs to a fibre $B$ and a smooth fibre $A$, we have 
$\s(L,1) = \s(L,x) \le \min(a,3b)$. Also, by Theorem
\ref{general-seshadri-constant}, $\s(L,1) \ge (0.93)\sqrt{L^2}$ or
$\s(L,1) = \min(a,3b)$.

If $b \le \frac{2a(0.93)^2}{9}$, then clearly $3b \le a$. Further 
$b \le \frac{2a(0.93)^2}{9} \Leftrightarrow  3b \le
(0.93)\sqrt{2ab}$. So $\s(L,1) = 3b$. 

If $b \ge \frac{a}{2(0.93)^2}$, then clearly $a \le 3b$. Further 
$b \ge \frac{a}{2(0.93)^2}\Leftrightarrow a \le (0.93)\sqrt{2ab}$. So
$\s(L,1) = a$. 

Finally, if $\frac{2a(0.93)^2}{9} \le b \le \frac{a}{2(0.93)^2}$, then
$a \ge (0.93)\sqrt{2ab}$ and $3b \ge (0.93)\sqrt{2ab}$. So $\s(L,1)
\ge (0.93)\sqrt{2ab}$. 
\end{proof}

\begin{theorem}\label{type6}
Let $X$ be a hyperelliptic surface of type 6. Let $L \equiv  (a,b)$ be an
ample line bundle on $X$. Then the following statements hold: 
\begin{enumerate}
\item  If $\frac{9\min(a,b)}{2(0.93)^2} \le
\rm{max}(a,b)$, then $\s(L,1) = 3\min(a,b)$. 
\item If $\frac{9\min(a,b)}{2(0.93)^2} \ge \max(a,b)$, then $\s(L,1) \ge
(0.93)\sqrt{L^2}$.
\end{enumerate}
\end{theorem}
\begin{proof}
Note that when $X$ has type 6, a fibre $B$ of $\Phi$ is given by
$(0,3)$ and a smooth fibre $A$ of $\Psi$ is given by $(3,0)$. So
$L\cdot A = 3b$ and $L \cdot B = 3a$. Since a very general point $x
\in X$ always belongs to a fibre $B$ and a smooth fibre $A$, we have 
$\s(L,1) = \s(L,x) \le \min(3a,3b)$. Also, by Theorem
\ref{general-seshadri-constant}, $\s(L,1) \ge (0.93)\sqrt{L^2}$ or
$\s(L,1) = \min(3a,3b)$. 

We have
\begin{eqnarray*}
\frac{9\min(a,b)}{2(0.93)^2} &\le& \rm{max}(a,b) \\
\Leftrightarrow  9(\min(a,b))^2 &\le& (0.93)^2(2ab) \\
\Leftrightarrow 3\min(a,b) &\le&  (0.93)\sqrt{2ab} \\
\Rightarrow\s(L,1) &=&  \min(3a,3b).
\end{eqnarray*}

On the other hand, 

$\frac{9\min(a,b)}{2(0.93)^2} \ge \max(a,b) \Leftrightarrow 
3\min(a,b) \ge (0.93)\sqrt{2ab} \Rightarrow \s(L,1) \ge (0.93)\sqrt{2ab}$.
\end{proof}

\begin{theorem}\label{type7}
Let $X$ be a hyperelliptic surface of type 7. Let $L \equiv (a,b)$ be an
ample line bundle on $X$. Then the following statements hold: 
\begin{enumerate}
\item  If $b \le \frac{a(0.93)^2}{18}$, then $\s(L,1) = 6b$. 
\item If $\frac{a(0.93)^2}{18} \le b \le \frac{a}{2(0.93)^2}$,
then $\s(L,1) \ge (0.93)\sqrt{L^2}$.
\item If $b \ge \frac{a}{2(0.93)^2}$, then $\s(L,1) = a$.
\end{enumerate}
\end{theorem}
\begin{proof}
Note that when $X$ has type 7, a fibre $B$ of $\Phi$ is given by
$(0,1)$ and a smooth fibre $A$ of $\Psi$ is given by $(6,0)$. So
$L\cdot A = 6b$ and $L \cdot B = a$. Since a very general point $x
\in X$ always belongs to a fibre $B$ and a smooth fibre $A$, we have 
$\s(L,1) = \s(L,x) \le \min(a,6b)$. Also, by Theorem
\ref{general-seshadri-constant}, $\s(L,1) \ge (0.93)\sqrt{L^2}$ or
$\s(L,1) = \min(a,6b)$.

If $b \le \frac{a(0.93)^2}{18}$, then clearly $6b \le a$. Further 
$b \le \frac{a(0.93)^2}{18} \Leftrightarrow  6b \le
(0.93)\sqrt{2ab}$. So $\s(L,1) = 6b$. 

If $b \ge \frac{a}{2(0.93)^2}$, then clearly $a \le 6b$. Further 
$b \ge \frac{a}{2(0.93)^2}\Leftrightarrow a \le (0.93)\sqrt{2ab}$. So
$\s(L,1) = a$. 

Finally, if $\frac{a(0.93)^2}{18} \le b \le \frac{a}{2(0.93)^2}$, then
$a \ge (0.93)\sqrt{2ab}$ and $6b \ge (0.93)\sqrt{2ab}$. So $\s(L,1)
\ge (0.93)\sqrt{2ab}$. 
\end{proof}

\begin{remark} \label{compare}
We compare the result in Theorem \ref{general-seshadri-constant} with some
  bounds in the literature.  There has been a lot of interest in
  finding good lower bound for $\s(L,1)$. See, for instance, 
\cite{N,HK,ST,SSyz,FSST}. 

Let $X$ be any surface and let $L$ be an ample line bundle on $X$. 
It is known that $\s(X,L,1) \ge \sqrt{\frac{7}{9}}\sqrt{L^2}$, or $X$
is fibred by Seshadri curves, or $X$ is a cubic surface in $\proj^3$
and $L = \str_X(1)$; see \cite[Corollary 3.3]{SSyz}.
Since $\sqrt{\frac{7}{9}}$ is approximately $0.88$, the bound we give
in Theorem \ref{general-seshadri-constant} is better. 

Another recent result in this direction is
contained in \cite{FSST}. Let $d:=L^2$ and suppose that $d$ is not a
square. Then the equation $y^2-dx^2=1$ is known as {\it Pell's
  equation}. If $x=p,y=q$ is a solution to this
equation, then \cite[Theorem 1.3]{FSST} shows that $\s(L,1) \ge
\frac{p}{q}d$ or $\s(X,L,1)$ is contained in a finite set
${\rm Exc}(d;p,q)$ of rational
numbers which are easy to list. Though this bound is often better than
$(0.93)\sqrt{L^2}$, the set ${\rm Exc}(d;p,q)$ is typically large.

As an example, let $X$ be a hyperelliptic surface of type 6 and let $L
\equiv (5,11)$. Then $d = L^2 = 110$ and $\sqrt{L^2} \sim 10.49$.
By Theorem \ref{type6}, $\s(X,L,1) \ge (0.93)\sqrt{110} \sim 9.75$. On
the other hand, $(2,21)$ is a solution to Pell's equation
$y^2-110x^2=1$. So by \cite[Theorem 1.3]{FSST}, $\s(X,L,1) \ge
\frac{220}{21} \sim 10.48$, or $\s(X,L,1) \in {\rm Exc}(110;2,21)$.
Though $10.48$ is a much better approximation to $\sqrt{L^2}$ compared to
our $9.75$, the exceptional set $ {\rm Exc}(110;2,21)$ is large and it
is not easy in general to lower the number of possibilities. 
In this case, ${\rm Exc}(110;2,21) =\{ 1,2,\ldots,10\} \cup \{
  \frac{r}{s} ~|~ 1 \le \frac{r}{s} < \frac{220}{21} {\rm ~and~} 2
                                \le s <  21^2=441\}$.

We also note that our results give precise values of $\s(X,L,1)$ in
many cases. For example, if $X$ is hyperelliptic of type 6 and 
$L\equiv (5,b)$ and $b \ge 27$, then $\s(X,L,1) =
15$, by 
Theorem \ref{type6}.

\end{remark}

{\bf Acknowledgement:} We thank the  
referee for carefully reading the paper and 
suggesting some changes which improved the exposition.

\bibliographystyle{plain}

\end{document}